\documentclass[reqno, 11pt]{amsart}
\usepackage{amsmath}

\begin{document}

\newtheorem{lem}{Lemma}[section]
\newtheorem{pro}[lem]{Proposition}
\newtheorem{thm}[lem]{Theorem}
\newtheorem{cor}[lem]{Corollary}

\newcommand{\bgamma}{\boldsymbol{\gamma}}
\newcommand{\bdelta}{\boldsymbol{\delta}}
\newcommand{\btheta}{\boldsymbol{\theta}}
\newcommand{\bsigma}{\boldsymbol{\sigma}}
\newcommand{\brho}{\boldsymbol{\rho}}


\title{Symplectic displacement energy for exact Lagrangian immersions}

\author{Manabu AKAHO}
\address{Department of Mathematics and Information Sciences,
Tokyo Metropolitan University, 
1-1 Minami-Ohsawa, Hachioji, Tokyo 192-0397, Japan}
\email{akaho@tmu.ac.jp}

\thanks{Supported by JSPS Grant-in-Aid for Young Scientists (B)}

\maketitle

\begin{abstract}
We give an inequality of the displacement energy for exact Lagrangian immersions 
and the symplectic area of punctured holomorphic discs. 
Our approach is based on Floer homology for Lagrangian immersions \cite{ak} and 
Chekanov's homotopy technique of continuations \cite{ch1}. 
\end{abstract}

\section{\bf Introduction}
Let $(M,\omega)$ be a symplectic manifold,
and $\iota:L\to M$ a Lagrangian immersion, i.e.
$\iota:L\to M$ is an immersion which satisfies
$\dim L=\dim M/2$ and $\iota^*\omega=0$.
We call a Lagrangian immersion $\iota:L\to M$ {\it exact}~if
\[
\int_{D^2}v^*\omega=0
\]
for any pair of smooth maps $v:D^2\to M$ and $\bar{v}:\partial D^2\to L$ 
such that $v|_{\partial D^2}=\iota\circ\bar{v}$,
where $D^2:=\{z\in\mathbb{C}:|z|\leq 1\}$.
Let $K(\iota)$ denote the set of the pairs of 
smooth maps $v:D^2\to M$ and $\bar{v}:[0,1]\to L$ such that: 
\begin{itemize}
\item $\bar{v}(0)\neq \bar{v}(1)$ and $\iota\circ\bar{v}(0)=\iota\circ\bar{v}(1)$,
\vspace{3pt}
\item $v|_{\partial D^2}=\iota\circ \bar{v}$, 
where we identify $e^{2\pi i\theta}\in\partial D^2$ with $\theta\in [0,1]$.
\end{itemize}
Then we define $\sigma$ by
\[
\sigma:=\inf
\left\{
\int_{D^2}v^*\omega :
(v,\bar{v})\in K(\iota) \mbox{ with } \int_{D^2}v^*\omega>0
\right\}.
\]
Note that $\sigma=\infty$ if $\int_{D^2}v^*\omega=0$ for any $(v,\bar{v})\in K(\iota)$.

A smooth function $H: [0,1]\times M\to \mathbb{R}$ defines 
the time-dependent Hamiltonian vector field $X_H$ on $M$ by $dH=\omega(X_H,\cdot)$,
and $\varphi^H_t$ denotes the Hamiltonian isotopy generated by $X_H$, i.e.
$\varphi^H_t:M\to M$ is given by
\[
\frac{d\varphi^H_t}{dt}=X_H\circ\varphi^H_t
\quad \mbox{and} \quad 
\varphi^H_0={\rm id}.
\]
We call $\varphi^H_1$ the time one map generated by $X_H$.
If $M$ is non-compact, we assume that $H$ is compactly supported. 
Then, following Hofer \cite{ho1}, we define a norm of $H$ by
\[
\| H\|:=\int_0^1\left(
\max_{x\in M}H(t,x)-\min_{x\in M}H(t,x)\right)dt.
\]

Our main theorem is the following:

\begin{thm}
Let $(M,\omega)$ be a closed symplectic manifold
or a non-compact symplectic manifold with convex end,
and $\iota:L\to M$ an exact Lagrangian immersion from a closed manifold $L$.
Suppose the non-injective points of $\iota: L\to M$ are transverse.
Let $\varphi^H_1$ be the time one map generated by $X_H$.
If $\|H\|<\sigma$, and 
$\iota:L\to M$ and $\varphi^H_1\circ \iota:L\to M$ intersect transversely, then 
\[
\sharp \left\{(x,x')\in L\times L: \iota(x)=(\varphi^H_1\circ\iota)(x')\right\}
\geq \sum_{k=0}^{\dim L}\dim H_k(L;\mathbb{Z}_2).
\]
\end{thm}

Since we use pseudoholomorphic curves, 
we put the convex end condition when $M$ is non-compact;
for example, the cotangent bundles of closed manifolds equipped with 
the canonical symplectic structure, 
the symplectic vector spaces $\mathbb{C}^n$ and so on.

Following Hofer \cite{ho2}, 
we define the {\it symplectic displacement energy} $e(A)$ for a subset $A\subset M$ by
\[
e(A):=\inf
\left\{\|H\| :
\begin{array}{l}
A\cap \varphi^H_1(A)=\emptyset, 
\mbox{ where } \varphi^H_1 \mbox{ is the time one}
\\
\mbox{map generated by }X_H
\end{array}
\right\}.
\]
Note that $e(A)=\infty$ if $A\cap \varphi^H_1(A)\neq \emptyset$ for any $\varphi^H_1$.

Since we may perturb $H$ to be generic so that 
$\iota:L\to M$ and $\varphi^H_1\circ \iota:L\to M$ intersect transversely, 
we obtain the following corollary:

\begin{cor}
For our exact Lagrangian immersion $\iota:L\to M$,
\[
\sigma\leq e(\iota(L)).
\]
\end{cor}

The symplectic displacement energy for Lagrangian {\it submanifolds} 
was early discussed by 
Polterovich \cite{po1}, Chekanov \cite{ch1}, \cite{ch2} and Oh \cite{oh1}.
We call a Lagrangian submanifold $L\subset M$ {\it rational} if 
\[
\left\{\int_{D^2}v^*\omega:v(D^2,\partial D^2)\to (M,L)\right\}
=\Sigma \mathbb{Z}
\]
for some $\Sigma>0$. 
Polterovich \cite{po1} proved the following theorem;
his approach is based on Gromov's theory of pseudoholomorphic curves:

\begin{thm}
For a rational Lagrangian submanifold $L\subset M$,
\[
\frac{1}{2}\Sigma\leq e(L).
\]
\end{thm} 

Moreover, Chekanov \cite{ch1} improved the Polterovich's theorem:

\begin{thm}
For a rational Lagrangian submanifold $L\subset M$,
\[
\Sigma\leq e(L).
\]
\end{thm}

In fact, he introduced a variant of Floer homology and obtained Theorem 1.4 
as a corollary of the following theorem (\cite{ch1}):

\begin{thm}
Let $L\subset M$ be a rational Lagrangian submanifold. If $\|H\|<\Sigma$, and
$L$ and $\varphi^H_1(L)$ intersect transversely, then
\[
\sharp \left(L\cap \varphi^H_1(L)\right)\geq \sum_{k=0}^{\dim L}H_k(L;\mathbb{Z}_2).
\]
\end{thm}

After that, Chekanov \cite{ch2} introduced some homological algebra and
relaxed the assumption of Theorem 1.5; 
and Oh \cite{oh1} used Gromov--Floer theory of perturbed Cauchy--Riemann equation and
simplified the proof of the inequality for the symplectic displacement energy of
Lagrangian submanifolds given in \cite{ch2}.

We observe Corollary 1.2.
Let $(M,\omega)=(\mathbb{R}^2,dx\wedge dy)$ be the 2-dimensional symplectic vector space,
and $\iota:L:=\{e^{i\theta}:0\leq\theta\leq2\pi\}\to \mathbb{R}^2$,
$e^{i\theta}\mapsto (\sin\theta\cos\theta,\sin\theta)$, 
an exact Lagrangian immersion ``{\it figure 8}."
Then $\sigma=\frac{2}{3}$.
On the other hand, from the Hofer--Zehnder capacity \cite{hz}, 
$e(\iota(L))\geq \frac{4}{3}(=2\sigma)$;
and moreover, the following $H$ attains $e(\iota(L))$ and its Hofer norm is $\frac{4}{3}$:
\[
H(x,y):=\begin{cases}
0 & y\leq -1,
\\
-\int_{-1}^y2t\sqrt{1-t^2}dt & -1\leq y\leq 0,
\\
\frac{2}{3}+\int_0^y2t\sqrt{1-t^2}dt & 0\leq y\leq 1,
\\
\frac{4}{3} & 1\leq y.
\end{cases}
\]
(Cut $H$ outside of a large disc to be compactly supported.)
Thus $\sigma=\frac{2}{3}<e(\iota(L))=\frac{4}{3}$.
The author does not know any non-trivial example which attains the equality of Corollary 1.2.
\\

{\it Acknowledgements.}
The author would like to thank K. Irie, H. Iriyeh, M. Kawasaki, 
H. Ohta, K. Ono and F. Ziltener for useful discussions,
and in particular K. Fukaya for unceasing warm encouragement.

\section{\bf Morse theory for Floer homology of exact Lagrangian immersions}

We prepare some notation and review Morse theory 
for our Floer homology of exact Lagrangian immersions.

Let $(M,\omega)$ be a symplectic manifold,
and $\iota:L\to M$ an exact Lagrangian immersion;
and let $H=H(t,x):[0,1]\times M\to \mathbb{R}$ be a smooth function.
Fix $s\in (0,1]$; and $X_{sH}$ is 
the time-$t$ dependent Hamiltonian vector field on $M$ given by $d(sH)=\omega(X_{sH},\cdot)$.
Note that $X_{sH}=sX_H$.
Let $\varphi^{sH}_t:M\to M$ be the Hamiltonian isotopy generated by $X_{sH}$, i.e.
$\varphi^{sH}_t:M\to M$ is given by
\[
\frac{d\varphi^{sH}_t}{dt}=X_{sH}\circ\varphi^{sH}_t
\quad \mbox{and} \quad
\varphi^{sH}_0={\rm id}.
\]

Fix a point $x_0\in L$. 
We define $\Omega$ to be the set of the pairs of smooth maps 
$\gamma:[0,1]\to M$ and $\bar{\gamma}:\{0,1\}\to L$ such that: 
\begin{itemize}
\item $\gamma(0)=\iota\circ\bar{\gamma}(0)$ and $\gamma(1)=\iota\circ\bar{\gamma}(1)$,
\item there is a pair of smooth maps $u:[0,1]\times[0,1]\to M$
and $\bar{u}:[0,1]\times\{0,1\}\to L$ such that:
\begin{itemize}
\item $u(\tau,0)=\iota\circ\bar{u}(\tau,0)$ and $u(\tau,1)=\iota\circ\bar{u}(\tau,1)$,
\item $u(0,t)=\iota(x_0)$, and $\bar{u}(0,0)=x_0$ and $\bar{u}(0,1)=x_0$,
\item $u(u,t)=\gamma(t)$, and $\bar{u}(1,0)=\bar{\gamma}(0)$ and $\bar{u}(1,1)=\bar{\gamma}(1)$.
\end{itemize}
\end{itemize}
Let $\bgamma$ denote $(\gamma,\bar{\gamma})\in\Omega$.
We may think of the tangent space $T_{\bgamma}\Omega$
as the set of the triples of a section $\xi$ of $\gamma^*TM$,
$\xi_0\in T_{\bar{\gamma}(0)}L$ and $\xi_1\in T_{\bar{\gamma}(1)}L$
such that $\xi(0)=\iota_*\xi_0$ and $\xi(1)=\iota_*\xi_1$;
for simplicity, we omit to write the $T_{\bar{\gamma}(0)}L$ and $T_{\bar{\gamma}(1)}L$ components
of $T_{\bgamma}\Omega$.

We define a functional $F_s:\Omega\to \mathbb{R}$ by
\[
F_s(\bgamma):=-\int_{[0,1]\times[0,1]}u^*\omega
-s\int_0^1H(t,\gamma(t))dt.
\]
Since $\iota:L\to M$ is exact Lagrangian, $F_s$ is independent of the choice of $(u,\bar{u})$,
i.e. $F_s$ depends only on $\bgamma$.
The differential $dF_s$ is given by
\[
(dF_s)_{\bgamma}(\xi)
=\int_0^1\omega
\left(\xi(t),-\frac{d\gamma(t)}{dt}+sX_H(\gamma(t))
\right)dt
\]
for $\xi\in T_{\bgamma}\Omega$.
Hence $\bgamma$ is a critical point of $F_s$ if and only if
$\gamma$ is a time-1 trajectory of $X_{sH}$
which starts and ends on $\iota(L)$.
We define $c_s$ to be the set of the critical points of $F_s$, i.e.
\[
c_s:=\left\{\bgamma:=(\gamma,\bar{\gamma})\in\Omega:
\frac{d\gamma(t)}{dt}=X_{sH}(\gamma(t))\right\}.
\]
Note that, let $\gamma(t):=\varphi^{sH}_t(\delta(t))$, 
then $(\gamma,\bar{\gamma})\in c_s$ if and only if
\[
\delta(t)\equiv p\in\iota(L)\cap (\varphi^{sH}_1)^{-1}(\iota(L)).
\]
Thus $\bgamma\in c_s$ gives an intersection point of $\iota(L)$ and 
$(\varphi^{sH}_1)^{-1}(\iota(L))$.

Let $\{J_t\}_{t\in[0,1]}$ be a time-dependent tame almost complex structure on $M$.
We define a Riemannian metric $G$ on $\Omega$ by
\[
G(\xi_1,\xi_2):=\int_0^1\omega(\xi_1(t),J_t\xi_2(t))dt
\]
for $\xi_1, \xi_2\in T_{\bgamma}\Omega$.
Then the gradient vector of $F_s$ with respect to $G$ is given~by
\[
(\nabla F_s)_{\bgamma}
=J_t(\gamma(t))\left(\frac{d\gamma(t)}{dt}-sX_H(\gamma(t))\right).
\]

Note that $\Omega$ and $F_s$ are essentially the same as used 
in Chekanov \cite{ch2} and Oh \cite{oh1}
but we modify them for exact Lagrangian immersions.

\section{\bf Floer homology for exact Lagrangian immersions}

We introduce a variant of Floer homology, inspired by Chekanov \cite{ch1} and \cite{ch2},
for exact Lagrangian immersions.
In this section we do {\it not} use $\|H\|$.

Let $(M,\omega)$ be a closed symplectic manifold or a non-compact symplectic manifold 
with convex end,
and $\iota:L\to M$ an exact Lagrangian immersion from a closed manifold $L$.
For generic $H$, there exists an open dense subset $T\subset[0,1]$ such that,
for $s\in T$, $(\varphi^{sH}_1\circ\iota)_*T_{\bar{\gamma}(0)}L$
and $\iota_*T_{\bar{\gamma}(1)}L$ intersect transversely in $T_{\gamma(1)}M$
for $(\gamma,\bar{\gamma})\in c_s$;
we always assume that $H$ is generic and take such $T\subset[0,1]$.
Note that $c_s$ is a finite set for $s\in T$.

Let $s\in T$. 
We define $\mathcal{M}_s(\bgamma,\bdelta)$ 
for $\bgamma:=(\gamma,\bar{\gamma}), \bdelta:=(\delta,\bar{\delta})\in c_s$
to be the set of 
the {\it (descending) gradient trajectories} $(u,\bar{u})$ of $F_s$ from $\bgamma$ to $\bdelta$,
i.e. the pairs of smooth maps $u:\mathbb{R}\times[0,1]\to M$
and $\bar{u}:\mathbb{R}\times\{0,1\}\to L$ such that: 
\begin{itemize}
\item $u(\tau,0)=\iota\circ\bar{u}(\tau,0)$ and $u(\tau,1)=\iota\circ\bar{u}(\tau,1)$,
\item $\lim_{\tau\to -\infty}u(\tau,t)=\gamma(t)$ 
and $\lim_{\tau\to-\infty}\bar{u}(\tau,i)=\bar{\gamma}(i)$ for $i=0,1$, 
\item $\lim_{\tau\to\infty}u(\tau,t)=\delta(t)$,
and $\lim_{\tau\to\infty}\bar{u}(\tau,i)=\bar{\delta}(i)$ for $i=0,1$, 
\item $u$ is a solution of the perturbed Cauchy--Riemann equation:
\[
\frac{\partial u(\tau,t)}{\partial \tau}
+J_t(u(\tau,t))\left(\frac{\partial u(\tau,t)}{\partial t}-sX_H(u(\tau,t))\right)=0.
\]
\end{itemize}
Note that $\mathbb{R}$ acts on $\mathcal{M}_s(\bgamma,\bdelta)$ 
by the translations of $\tau$,
and let $\hat{\mathcal{M}}_s(\bgamma,\bdelta)$ denote the quotient.
Since the boundary value $\iota\circ\bar{u}$ does not switch 
sheets at non-injective points of the immersion,
we can use the usual local theory of the perturbed Cauchy--Riemann equation.
Hence we have the following theorem (\cite{ak}, \cite{fl1}, \cite{fl2} and \cite{oh1}):

\begin{thm}
For generic $\{J_t\}_{t\in[0,1]}$,
$\hat{\mathcal{M}}_s(\bgamma,\bdelta)$ is a finite dimensional smooth manifold.
\end{thm}

Let $\hat{\mathcal{M}}_s^k(\bgamma,\bdelta)$ denote
the $k$-dimensional component of 
$\hat{\mathcal{M}}_s(\bgamma,\bdelta)$.

Following \cite{oh1} and \cite{sa1},
we define the energy $E(u)$ of 
$(u,\bar{u})\in\hat{\mathcal{M}}_s(\bgamma,\bdelta)$~by
\begin{eqnarray*}
E(u)&:=&\int_{-\infty}^{\infty}\int_0^1
\left|\frac{\partial u(\tau,t)}{\partial \tau}\right|_{J_t}^2 dtd\tau
\\
&=&
\int_{-\infty}^{\infty}\int_0^1
\omega\left(
\frac{\partial u(\tau,t)}{\partial \tau},J_t(u(\tau,t))\frac{\partial u(\tau,t)}{\partial \tau}
\right)
dtd\tau.
\end{eqnarray*}

\begin{lem}
For $(u,\bar{u})\in\hat{\mathcal{M}}_s(\bgamma,\bdelta)$, 
\[
E(u)=F_s(\bgamma)-F_s(\bdelta).
\]
\end{lem}

\begin{proof}
Since $u$ satisfies the perturbed Cauchy--Riemann equation,
\begin{eqnarray*}
E(u)&=&\int_{-\infty}^{\infty}\int_0^1
\omega\left(\frac{\partial u(\tau,t)}{\partial \tau},\frac{\partial u(\tau,t)}{\partial t}-sX_H(u(\tau,t))
\right)
\\
&=&
\int_{\mathbb{R}\times[0,1]}u^*\omega
+s\int_{-\infty}^{\infty}\int_0^1\frac{\partial H(t,u(\tau,t))}{\partial \tau}dtd\tau
\\
&=&
\int_{\mathbb{R}\times[0,1]}u^*\omega
+s\int_0^1H(t,\delta(t))dt-s\int_0^1H(t,\gamma(t))dt
\\
&=&
F_s(\bgamma)-F_s(\bdelta).
\end{eqnarray*}
\end{proof}

Suppose $\sigma>0$.
We take $0<\kappa<\sigma$, and choose an interval $[b_-,b_+)\subset \mathbb{R}$ such that
\[
b_-<0<b_+ 
\quad \mbox{and} \quad
b_+-b_-=\kappa.
\]
Then we define a function $f_s:c_s \to [b_-,b_+)$ by
\[
f_s(\bgamma)\equiv F_s(\bgamma) \mod \kappa
\]
for $\bgamma\in c_s$.
Following Chekanov \cite{ch1},
we call $(u,\bar{u})\in\hat{\mathcal{M}}_s(\bgamma,\bdelta)$ 
a {\it distinguished} gradient trajectory if 
\[
F_s(\bgamma)-F_s(\bdelta)=f_s(\bgamma)-f_s(\bdelta).
\]
We define $\hat{\mathcal{M}}_s^d(\bgamma,\bdelta)$
to be the set of the distinguished gradient trajectories in 
$\hat{\mathcal{M}}_s(\bgamma,\bdelta)$; 
in fact, since $\iota:L\to M$ is exact Lagrangian,
\[
\hat{\mathcal{M}}_s^d(\bgamma,\bdelta)
=
\begin{cases}
\hat{\mathcal{M}}_s(\bgamma,\bdelta) & \mbox{if } 
F_s(\bgamma)-F_s(\bdelta)=f_s(\bgamma)-f_s(\bdelta),
\\
\emptyset & \mbox{otherwise}.
\end{cases}
\]
Let $\hat{\mathcal{M}}_s^{d,k}(\bgamma,\bdelta)$ denote
the $k$-dimensional component of 
$\hat{\mathcal{M}}_s^d(\bgamma,\bdelta)$.

From Lemma 3.2, we have the following lemma:
\begin{lem}
For $(u,\bar{u})\in \hat{\mathcal{M}}_s^d(\bgamma,\bdelta)$,
\[
E(u)<\kappa.
\]
\end{lem}

Assume that the non-injective points of $\iota:L\to M$ are transverse.

\begin{pro}
Let $s_i\to s_{\infty}\in \{0\}\cup T$, 
and $(u_i,\bar{u}_i)\in\hat{\mathcal{M}}_{s_i}(\bgamma_i,\bdelta_i)$
be a sequence of gradient trajectories with $E(u_i)<\kappa$.
Then $\{(u_i,\bar{u}_i)\}$ has a subsequence which
converges to a broken gradient trajectory without bubble tree
in the sense of Floer--Gromov convergence.
\end{pro}

\begin{proof}
Since $L$ is compact, taking a subsequence if necessary,
$\bgamma_i$ and $\bdelta_i$ converge to
$\bgamma\in c_{s_{\infty}}$ and $\bdelta\in c_{s_{\infty}}$, respectively.
Then, by Lemma 3.2, $E(u_i)$ is uniformly bounded, 
and the Floer--Gromov compactness theorem 
(\cite{fl1}, \cite{fl2} and \cite{sa1}) implies that $\{(u_i,\bar{u}_i)\}$ has a subsequence 
which converges to a broken gradient trajectory 
$((v_1,\bar{v}_1),\ldots,(v_N,\bar{v}_N)) 
\in\hat{\mathcal{M}}_{s_{\infty}}(\bgamma,\btheta_1)
\times\cdots\times\hat{\mathcal{M}}_{s_{\infty}}(\btheta_{N-1},\bdelta)$
with bubble trees;
the {\it tail} components of the bubble trees~are
\begin{itemize}
\item[(i)] a pseudoholomorphic sphere $v:S^2\to M$, 
\item[(ii)] a pseudoholomorphic disc
$v:D^2\to M$ with $\bar{v}:\partial D^2\to L$ such that $v|_{\partial D^2}=\iota\circ \bar{v}$,
\item[(iii)] a pseudoholomorphic disc $v:D^2\to M$ of $(v,\bar{v})\in K(\iota)$.
\end{itemize}
But, since our Lagrangian immersion is exact, the bubbles of (i) and (ii) can not occur.
Moreover, since the symplectic area of the bubble trees 
is less than or equal to $\limsup E(u_i)\leq \kappa \ (<\sigma)$ (\cite{sa1}),
the bubbles of (iii) can not occur.
Hence there is no bubble tree, 
and the subsequence converges to the broken gradient trajectory.
\end{proof}

To define our Floer homology, we use the following compactness theorems:

\begin{thm} 
$\hat{\mathcal{M}}_s^{d,0}(\bgamma,\bdelta)$ is compact.
\end{thm}

\begin{proof}
Suppose on the contrary that $\hat{\mathcal{M}}_s^{d,0}(\bgamma,\bdelta)$ is not compact.
Then there exists a sequence 
$\{(u_i,\bar{u}_i)\}\subset \hat{\mathcal{M}}_s^{d,0}(\bgamma,\bdelta)$ 
such that any subsequence does not converge in 
$\hat{\mathcal{M}}_s^{d,0}(\bgamma,\bdelta)$.
On the other hand,
from Lemma 3.3 and Proposition 3.4,
taking a subsequence if necessary,
$\{(u_i,\bar{u}_i)\}$ converges to 
a broken gradient trajectory
$((v_1,\bar{v}_1),\ldots,(v_N,\bar{v}_N))$ without bubble tree.
In this case, $N$ turns out to be 1 
and $(v_1,\bar{v}_1)\in\hat{\mathcal{M}}_s^0(\bgamma,\bdelta)$
for generic $\{J_t\}_{t\in[0,1]}$
because of the virtual dimension counting.
Since the subsequence preserves the condition 
$F_s(\bgamma)-F_s(\bdelta)=f_s(\bgamma)-f_s(\bdelta)$,
the limit $(v_1,\bar{v}_1)$ is in $\hat{\mathcal{M}}_s^{d,0}(\bgamma,\bdelta)$,
which contradicts that any subsequence does not converge in 
$\hat{\mathcal{M}}_s^{d,0}(\bgamma,\bdelta)$. 
Thus $\hat{\mathcal{M}}_s^{d,0}(\bgamma,\bdelta)$ is compact.
\end{proof}

\begin{thm} 
$\hat{\mathcal{M}}_s^{d,1}(\bgamma,\bdelta)$ 
has a suitable compactification whose boundary is given by
\[
\bigcup_{\btheta\in c_s}
\hat{\mathcal{M}}_s^{d,0}(\bgamma,\btheta)
\times \hat{\mathcal{M}}_s^{d,0}(\btheta,\bdelta).
\]
\end{thm}

\begin{proof}
The proof is based on the standard gluing-compactness argument in \cite{fl1}, \cite{fl2} and \cite{sa1}.
For a pair $((u_1,\bar{u}_1),(u_2,\bar{u}_2))
\in \hat{\mathcal{M}}_s^{d,0}(\bgamma,\btheta)\times\hat{\mathcal{M}}_s^{d,0}(\btheta,\bdelta)$,
the gluing procedure gives a unique connected component of 
$\hat{\mathcal{M}}_s^1(\bgamma,\bdelta)$, say $\mathcal{M}$,
such that the pair $((u_1,\bar{u}_1),(u_2,\bar{u}_2))$ is a compactifying point of $\mathcal{M}$.
Since 
\begin{eqnarray*}
F_s(\bgamma)-F_s(\bdelta)&=&F_s(\bgamma)-F_s(\btheta)+F_s(\btheta)-F_s(\bdelta)
\\
&=&f_s(\bgamma)-f_s(\btheta)+f_s(\btheta)-f_s(\bdelta)
\\
&=&f_s(\bgamma)-f_s(\bdelta),
\end{eqnarray*}
$\mathcal{M}$ is contained in $\hat{\mathcal{M}}_s^{d,1}(\bgamma,\bdelta)$.

On the other hand, 
from Lemma 3.3 and Proposition 3.4,
let $\{(u_i,\bar{u}_i)\}\subset \hat{\mathcal{M}}_s^{d,1}(\bgamma,\bdelta)$
be a sequence which converges 
to a broken gradient trajectory 
$((v_1,\bar{v}_1),\ldots,(v_N,\bar{v}_N))$ without bubble tree.
In this case, $N$ turns out to be 2 and
$((v_1,\bar{v}_1),(v_2,\bar{v}_2)) 
\in
\hat{\mathcal{M}}_s^0(\bgamma,\btheta)\times\hat{\mathcal{M}}_s^0(\btheta,\bdelta)$
for generic $\{J_t\}_{t\in [0,1]}$
because of the virtual dimension counting.
Since $F_s(\bgamma)-F_s(\bdelta)=f_s(\bgamma)-f_s(\bdelta)$, 
there exists $m\in\mathbb{Z}$ 
such that
\begin{eqnarray}
F_s(\bgamma)-F_s(\btheta)&=&f_s(\bgamma)-f_s(\btheta)+m\kappa,
\\
F_s(\btheta)-F_s(\bdelta)&=&f_s(\btheta)-f_s(\bdelta)-m\kappa.
\end{eqnarray}
From (1), since $0<E(u_1)=F_s(\bgamma)-F_s(\btheta)$ and 
$f_s(\bgamma)-f_s(\btheta)<\kappa$,
we obtain $0\leq m$;
and from (2), 
since $0<E(u_2)=F_s(\btheta)-F_s(\bdelta)$ and $f_s(\btheta)-f_s(\bdelta)<\kappa$,
we obtain $m\leq 0$.
Thus $m=0$, which implies $((v_1,\bar{v}_1),(v_2,\bar{v}_2))\in
\hat{\mathcal{M}}_s^{d,0}(\bgamma,\btheta)\times\hat{\mathcal{M}}_s^{d,0}(\btheta,\bdelta)$.
We obtain the compactification of $\hat{\mathcal{M}}_s^{d,1}(\bgamma,\bdelta)$.
\end{proof}

Analogous to the Floer homology of Lagrangian submanifolds (\cite{fl1}),
Theorem 3.4 allows us to define our Floer complex.
Let $C_s$ be the free $\mathbb{Z}_2$-module
\[
C_s:=\bigoplus_{\bgamma\in c_s}\mathbb{Z}_2\bgamma
\]
and we define a linear map $\partial_s: C_s\to C_s$ by
\[
\partial_s \bgamma:=\sum_{\bdelta\in c_s}\sharp\hat{\mathcal{M}}_s^{d,0}(\bgamma,\bdelta)\bdelta
\]
for $\bgamma\in c_s$.
Then Theorem 3.6 implies the following theorem (\cite{fl1}):

\begin{thm}
$\partial_s\circ\partial_s=0$.
\end{thm}

Let $H(C_s,\partial_s)$ denote the homology of $(C_s,\partial_s)$,
which is our {\it distinguished} Floer homology for exact Lagrangian immersions.

Next, we prepare some notation.
Let $0_L$ be the zero section of $T^*L$, 
and we fix a diffeomorphism $i_L:L\to 0_L$.
Take a small tubular neighborhood $U$ of $0_L$ in $T^*L$
and an immersion $\pi:U\to M$ such that: 
\begin{itemize}
\item $\pi\circ i_L=\iota$,
\item $\pi^*\omega$ equals the canonical symplectic form on $T^*L$.
\end{itemize}
We define a smooth function $sH\circ\pi:[0,1]\times U\to \mathbb{R}$ 
by $(sH\circ\pi)(t,x):=sH(t,\pi(x))$.
Take $s$ to be small so that for any $(\gamma,\bar{\gamma})\in c_s$
the image of $\gamma$ is contained in $\pi(U)$.
Then we divide $c_s$ into the following two sets $a_s$ and~$b_s$:
\[
a_s:=\left\{(\gamma,\bar{\gamma})\in c_s:
\begin{array}{l}
\mbox{there exists }\alpha:[0,1]\to U\mbox{ such that }\alpha(0),\alpha(1)\in 0_L
\\
\mbox{and }
\gamma=\pi\circ\alpha,
\mbox{ and }
\alpha(i)=i_L(\bar{\gamma}(i))
\mbox{ for }i=0,1
\end{array}
\right\},
\]
and $b_s:=\left\{(\gamma,\bar{\gamma})\in c_s:(\gamma,\bar{\gamma})\notin a_s\right\}$.
Note that, for small $s$,
$\bar{\gamma}(1)$ and $\bar{\gamma}(0)$ are very close in $L$ 
for $(\gamma,\bar{\gamma})\in a_s$;
on the other hand, for $(\gamma,\bar{\gamma})\in b_s$,
there exists $(y,y')\in L\times L$ such that $y\neq y'$, $\iota(y)=\iota(y')$
and $(\bar{\gamma}(0),\bar{\gamma}(1))$ is very close to $(y,y')$ in $L\times L$. 

\begin{thm}
There exists $s_0$ such that, for $s<s_0$,
$\{(u,\bar{u})\in \hat{\mathcal{M}}_s(\bgamma,\bdelta):E(u)<\kappa\}=\emptyset$
for any $\bgamma\in a_s$ and $\bdelta\in b_s$.
\end{thm}

\begin{proof}
Suppose on the contrary that there is no such $s_0$.
Then there exists a sequence $s_i\to 0$
such that 
$\{(u,\bar{u})\in\hat{\mathcal{M}}_{s_i}(\bgamma_i,\bdelta_i):E(u)<\kappa\}\neq\emptyset$
for some $\bgamma_i:=(\gamma_i,\bar{\gamma}_i)\in a_{s_i}$ and 
$\bdelta_i:=(\delta_i,\bar{\delta}_i)\in b_{s_i}$;
moreover, taking a subsequence if necessary, there exist
$(x,x), (y,y')\in L\times L$ such that:
\begin{itemize}
\item 
$(\bar{\gamma}_i(0),\bar{\gamma}_i(1))\to (x,x)$ and $\gamma_i(t)\to\iota(x)$,
\item 
$y\neq y'$ and $\iota(y)=\iota(y')$,
and $(\bar{\delta}_i(0),\bar{\delta}_i(1))\to (y,y')$ and $\delta_i(t)\to\iota(y)=\iota(y')$. 
\end{itemize}
Let $(u_i,\bar{u}_i)\in\hat{\mathcal{M}}_{s_i}(\bgamma_i,\bdelta_i)$ with $E(u_i)<\kappa$. 
From Proposition 3.4,
taking a subsequence if necessary, $(u_i,\bar{u}_i)$ converges to 
a broken {\it pseudoholomorphic strip} $((v_1,\bar{v}_1),\ldots,(v_N,\bar{v}_N))$
without bubble tree.
Since $(x,x)$ is an injective point and $(y,y')$ is a non-injective point, 
at least one of the broken components $(v_i,\bar{v}_i)$ is an element of $K(\iota)$.
But, since the symplectic area of the broken pseudoholomorphic strip
is less than or equal to
$\limsup E(u_i)\leq \kappa\ (<\sigma)$, there is no such $v_i$,
which contradicts to the existence of such $s_i\to 0$.
Thus there exists $s_0$ such that, for $s<s_0$, 
$\{(u,\bar{u})\in\hat{\mathcal{M}}_s(\bgamma,\bdelta):E(u)<\kappa\}=\emptyset$
for any $\bgamma\in a_s$ and $\bdelta\in b_s$.
\end{proof}

From Lemma 3.3, we obtain the following corollary:

\begin{cor}
There exists $s_0$ such that, for $s<s_0$,
$\hat{\mathcal{M}}_s^d(\bgamma,\bdelta)=\emptyset$
for any $\bgamma\in a_s$ and $\bdelta\in b_s$.
\end{cor}

Similarly, we can prove the following theorem and corollary: 

\begin{thm}
There exists $s_0$ such that, for $s<s_0$,  
$\{(u,\bar{u})\in\hat{\mathcal{M}}_s(\bgamma,\bdelta):E(u)<\kappa\}=\emptyset$
for any $\bgamma\in b_s$ and $\bdelta\in a_s$.
\end{thm}

\begin{cor}
There exists $s_0$ such that, for $s<s_0$,
$\hat{\mathcal{M}}_s^d(\bgamma,\bdelta)=\emptyset$
for any $\bgamma\in b_s$ and $\bdelta\in a_s$.
\end{cor}

Now we define the free $\mathbb{Z}_2$-modules
\[
A_s:=\bigoplus_{\bgamma\in a_s}\mathbb{Z}_2\bgamma
\quad \mbox{and} \quad 
B_s:=\bigoplus_{\bgamma\in b_s}\mathbb{Z}_2\bgamma.
\]
Note that $C_s=A_s\oplus B_s$.
From Corollary 3.9 and 3.11, the boundary operator $\partial_s$ has no cross term when $s<s_0$,
and we obtain the following corollary:

\begin{cor}
There exists $s_0$ such that, for $s<s_0$, 
$\partial_sA_s\subset A_s$ and $\partial_sB_s\subset B_s$.
\end{cor}

For $s<s_0$, let $H(A_s,\partial_s)$ and $H(B_s,\partial_s)$ denote
the homologies of $(A_s,\partial_s)$ and $(B_s,\partial_s)$, respectively.

Since $L$ is compact, there exists a function $\varepsilon:[0,1]\to[0,\infty)$ such that
$\lim_{s\to 0}\varepsilon(s)=0$ and $|F_s(\bgamma)|<\varepsilon(s)$ for $\bgamma\in a_s$.
In particular, $F_s(\bgamma)=f_s(\bgamma)$ for $\bgamma\in a_s$
when $\varepsilon(s)\leq \min\{b_+,-b_-\}$. 
We often use this $\varepsilon:[0,1]\to[0,\infty)$.

\begin{pro}
There exists $s_0$ such that for $s<s_0$
the image of $u$ is contained in $\pi(U)$
for any $\bgamma,\bdelta\in a_s$ and $(u,\bar{u})\in\hat{\mathcal{M}}_s(\bgamma,\bdelta)$.
\end{pro}

\begin{proof}
First, there exist $s_0$ and $C$ such that, for $s<s_0$,
$\sup_{(\tau,t)}|du(\tau,t)|_{J_t}<C$ for any $\bgamma,\bdelta\in a_s$ and
$(u,\bar{u})\in\hat{\mathcal{M}}_s(\bgamma,\bdelta)$.
(Suppose on the contrary that there is no such $s_0$ nor $C$.
Then there exist sequences $s_i\to 0$, 
$(u_i,\bar{u}_i)\in \hat{\mathcal{M}}_{s_i}(\bgamma_i,\bdelta_i)$
for some $\bgamma_i,\bdelta_i\in a_{s_i}$, and $(\tau_i,t_i)$ 
such that $|du_i(\tau_i,t_i)|\to\infty$.
Then, take a subsequence if necessary, 
there appear {\it non-trivial} bubble trees by the rescaling argument \cite{ms}.
On the other hand, $E(u_i)=F_{s_i}(\bgamma_i)-F_{s_i}(\bdelta_i)<2\varepsilon(s_i)\to 0$,
which contradicts that
the symplectic area of the non-trivial bubble trees is greater than 0.)
Let $D(z_0;r):=\{z\in \mathbb{C}:|z-z_0|<r\}\subset \mathbb{R}\times(0,1)$.
Following the mean value inequality of \cite{ms}, there exists $\hbar$ such that, 
if $u$ is a solution of the perturbed Cauchy--Riemann equation and
$\int_{D(z_0;r)}|du|_{J_t}^2<\hbar$, 
then $|du(z_0)|_{J_t}\leq \frac{8}{\pi r^2}\int_{D(z_0;r)}|du|_{J_t}^2$.
Take $s_1$ such that $2\varepsilon(s)<\hbar$ for $s<s_1$. 
Since $E(u)<2\varepsilon(s)$, we obtain
\[
|du(z_0)|_{J_t}^2\leq\frac{8}{\pi \varepsilon(s)^{1/2}}
\int_{D(z_0;\varepsilon(s)^{1/4})}|du|_{J_t}^2
\leq\frac{16}{\pi}\varepsilon(s)^{1/2}
\]
for any $\bgamma,\bdelta\in a_s$ and $(u,\bar{u})\in \hat{\mathcal{M}}_s(\bgamma,\bdelta)$.
Thus for $s<\min\{s_0,s_1\}$
\begin{eqnarray*}
\int_0^{t_0}\left|\frac{\partial u(\tau,t)}{\partial t}\right|_{J_t}dt
&\leq&\int_{[0,\varepsilon(s)^{1/4}]\cup[1-\varepsilon(s)^{1/4},1]}
\left|\frac{\partial u(\tau,t)}{\partial t}\right|_{J_t}dt
\\
&&
+\int_{[\varepsilon(s)^{1/4},1-\varepsilon(s)^{1/4}]}
\left|\frac{\partial u(\tau,t)}{\partial t}\right|_{J_t}dt
\\
&\leq&
2C\varepsilon(s)^{1/4}+\frac{4}{\sqrt{\pi}}(1-2\varepsilon(s)^{1/4})\varepsilon(s)^{1/4}
\\
&\to &0\ (\mbox{as}\ s\to 0)
\end{eqnarray*}
for any $\bgamma,\bdelta\in a_s$ and $(u,\bar{u})\in\hat{\mathcal{M}}_s(\bgamma,\bdelta)$,
which implies that the image of $u$ is contained in $\pi(U)$ when $s$ is small.
\end{proof}

Finally, we have the following theorem:

\begin{thm}
There exists $s_0$ such that for $s<s_0$  
\[
H(A_s,\partial_s)\cong \bigoplus_{k=0}^{\dim L}H_k(L;\mathbb{Z}_2).
\] 
\end{thm}

\begin{proof}
By Proposition 3.13, there exists $s_0$ such that for $s<s_0$
the image of $u$ is contained in $\pi(U)$
for any $\bgamma,\bdelta\in a_s$ and $(u,\bar{u})\in\hat{\mathcal{M}}_s^d(\bgamma,\bdelta)$.
In this case, $(A_s,\partial_s)$ agrees with the usual Floer complex generated by
the time-1 trajectories of $X_{sH\circ \pi}$ which start and end on $0_L$ in $T^*L$.
Thus $H(A_s,\partial_s)$ is isomorphic to 
$\bigoplus_{k=0}^{\dim L}H_k(L;\mathbb{Z}_2)$ (\cite{ak} and \cite{fl3}).
\end{proof}

\section{\bf Continuations}

Let $\rho:\mathbb{R}\to [0,1]$ be a smooth function such that
for some $R>0$ $\rho(\tau)=s_-$ when $\tau<-R$ and $\rho(\tau)=s_+$ when $\tau>R$.
We call such $\rho$ a {\it continuation function}.
In particular, for $0<s\leq S\leq 1$,
let $\rho_+$ be a {\it non-decreasing} continuation function such that
$\rho_+(\tau)=s$ when $\tau<-R$ and $\rho_+(\tau)=S$ when $\tau>R$,
and $\rho_-$ a {\it non-increasing} continuation function such that
$\rho_-(\tau)=S$ when $\tau<-R$ and $\rho_-(\tau)=s$ when $\tau>R$.

Let $s_-,s_+\in T$.
We define $\mathcal{M}_{\rho}(\bgamma,\bdelta)$ 
for $\bgamma=(\gamma,\bar{\gamma})\in c_{s_-}$ and 
$\bdelta=(\delta,\bar{\delta})\in c_{s_+}$
to be the set of 
the {\it continuation trajectories} $(u,\bar{u})$ from $\bgamma$ to $\bdelta$, i.e.
the pairs of smooth maps $u:\mathbb{R}\times[0,1]\to M$
and $\bar{u}:\mathbb{R}\times \{0,1\}\to L$ such that: 
\begin{itemize}
\item $u(\tau,0)=\iota\circ\bar{u}(\tau,0)$ and $u(\tau,1)=\iota\circ\bar{u}(\tau,1)$,
\item $\lim_{\tau\to -\infty}u(\tau,t)=\gamma(t)$ 
and $\lim_{\tau\to-\infty}\bar{u}(\tau,i)=\bar{\gamma}(i)$ for $i=0,1$,
\item $\lim_{\tau\to\infty}u(\tau,t)=\delta(t)$ 
and $\lim_{\tau\to\infty}\bar{u}(\tau,i)=\bar{\delta}(i)$ for $i=0,1$,
\item $u$ is a solution of the perturbed Cauchy--Riemann equation:
\[
\frac{\partial u(\tau,t)}{\partial \tau}+J_t(u(\tau,t))
\left(
\frac{\partial u(\tau,t)}{\partial t}-\rho(\tau)X_H(u(\tau,t))
\right)=0.
\]
\end{itemize}
In this case, $\mathbb{R}$ does {\it not} act on $\mathcal{M}_{\rho}(\bgamma,\bdelta)$
when $\rho(\tau)$ is not constant.
Since the boundary value $\iota\circ\bar{u}$ does not switch 
sheets at non-injective points of the immersion,
we can use the usual local theory of the perturbed Cauchy--Riemann equation.
Hence we have the following theorem:

\begin{thm}
For generic $\{J_t\}_{t\in [0,1]}$,
$\mathcal{M}_{\rho}(\bgamma,\bdelta)$ is a finite dimensional smooth manifold.
\end{thm}

Let $\mathcal{M}_{\rho}^k(\bgamma,\bdelta)$ denote the $k$-dimensional component of 
$\mathcal{M}_{\rho}(\bgamma,\bdelta)$.

We define the energy $E(u)$ of $(u,\bar{u})\in\mathcal{M}_{\rho}(\bgamma,\bdelta)$ by
\begin{eqnarray*}
E(u)&:=&\int_{-\infty}^{\infty}\int_0^1
\left|\frac{\partial u(\tau,t)}{\partial \tau}\right|_{J_t}^2 dtd\tau
\\
&=&
\int_{-\infty}^{\infty}\int_0^1
\omega\left(
\frac{\partial u(\tau,t)}{\partial \tau},J_t(u(\tau,t))\frac{\partial u(\tau,t)}{\partial \tau}
\right)
dtd\tau.
\end{eqnarray*}

\begin{lem}
For $(u,\bar{u})\in\mathcal{M}_{\rho}(\bgamma,\bdelta)$, 
\[
E(u)=F_{s_-}(\bgamma)-F_{s_+}(\bdelta)
-\int_{-\infty}^{\infty}\frac{d\rho(\tau)}{d\tau}\int_0^1H(t,u(\tau,t))dtd\tau.
\]
\end{lem}

\begin{proof}
Since $u$ satisfies the perturbed Cauchy--Riemann equation,
\begin{eqnarray*}
E(u)&=&\int_{-\infty}^{\infty}\int_0^1
\omega\left(\frac{\partial u(\tau,t)}{\partial \tau},\frac{\partial u(\tau,t)}{\partial t}-\rho(\tau)X_H(u(\tau,t))
\right)
\\
&=&
\int_{\mathbb{R}\times[0,1]}u^*\omega
+\int_{-\infty}^{\infty}\rho(\tau)
\int_0^1\frac{\partial H(t,u(\tau,t))}{\partial \tau}dtd\tau
\\
&=&
\int_{\mathbb{R}\times[0,1]}u^*\omega
+s_+\int_0^1H(t,\delta(t))dt-s_-\int_0^1H(t,\gamma(t))dt
\\
&&
\quad \quad \quad \quad \quad \quad \quad 
-\int_{-\infty}^{\infty}\frac{d\rho(\tau)}{d\tau}\int_0^1H(t,u(\tau,t))dtd\tau
\\
&=&
F_{s_-}(\bgamma)-F_{s_+}(\bdelta)
-\int_{-\infty}^{\infty}\frac{d\rho(\tau)}{d\tau}\int_0^1H(t,u(\tau,t))dtd\tau.
\end{eqnarray*}
\end{proof}

We call $(u,\bar{u})\in \mathcal{M}_{\rho}(\bgamma,\bdelta)$ a {\it distinguished} continuation 
trajectory if
\[
F_{s_-}(\bgamma)-F_{s_+}(\bdelta)
=
f_{s_-}(\bgamma)-f_{s_+}(\bdelta),
\]
and define $\mathcal{M}_{\rho}^d(\bgamma,\bdelta)$ to be the set of the distinguished 
continuation trajectories in $\mathcal{M}_{\rho}(\bgamma,\bdelta)$; in fact,
since $\iota:L\to M$ is exact Lagrangian, 
\[
\mathcal{M}_{\rho}^d(\bgamma,\bdelta)
=
\begin{cases}
\mathcal{M}_{\rho}(\bgamma,\bdelta) &
\mbox{if } F_{s_-}(\bgamma)-F_{s_+}(\bdelta)
=
f_{s_-}(\bgamma)-f_{s_+}(\bdelta),
\\
\emptyset & \mbox{otherwise}.
\end{cases}
\]
Let $\mathcal{M}_{\rho}^{d,k}(\bgamma,\bdelta)$ denote
the $k$-dimensional component of 
$\mathcal{M}_{\rho}^d(\bgamma,\bdelta)$.

If $M$ is non-compact, we assume that
$H:[0,1]\times M\to \mathbb{R}$ is compactly supported.
Then
\[
\int_0^1\min_{x\in M}H(t,x)dt
\leq 0\leq
\int_0^1\max_{x\in M}H(t,x)dt.
\]
On the other hand, if $M$ is closed, 
since $H(t,x)$ and $H(t,x)+c$ for $c\in\mathbb{R}$ give the same Hamiltonian vector 
field, we may assume
\[
\int_0^1\min_{x\in M}H(t,x)dt
\leq 0\leq
\int_0^1\max_{x\in M}H(t,x)dt.
\]
Define $a_-$ and $a_+$ by
\[
a_-:=-\int_0^1\max_{x\in M}H(t,x)dt
\quad \mbox{and} \quad
a_+:=-\int_0^1\min_{x\in M}H(t,x)dt.
\]
Note that $\|H\|=a_+-a_-$.
Suppose $\|H\|<\sigma$, 
and we choose $\|H\|<\kappa<\sigma$ and 
the interval $[b_-,b_+)\subset \mathbb{R}$ such that
\[
b_-<a_-\leq 0\leq a_+<b_+
\quad \mbox{and} \quad
b_+-b_-=\kappa.
\]
Note that we use $\|H\|<\sigma$ here.

\begin{lem}
There exists $s_0$ such that for $s<s_0$
\[
E(u)<\kappa
\]
for $\bgamma\in a_s$ and $(u,\bar{u})\in \mathcal{M}_{\rho_+}^d(\bgamma,\bdelta)$.
\end{lem}

\begin{proof}
By Lemma 4.2,
since $\rho_+$ is non-decreasing,  we obtain
\[
E(u)\leq F_s(\bgamma)-F_S(\bdelta)+(S-s)a_+.
\]
Note that there exists $s_0$ such that, for $s<s_0$, 
$\varepsilon(s)<b_+-a_+$ and $|f_s(\bgamma)|<\varepsilon(s)$ for $\bgamma\in a_s$.
Since $(u,\bar{u})$ is distinguished and $b_-\leq f_S(\bdelta)$, for $s<s_0$,
\begin{eqnarray*}
E(u)&\leq& F_s(\bgamma)-F_S(\bdelta)+(S-s)a_+
\\
&=& f_s(\bgamma)-f_S(\bdelta)+(S-s)a_+
\\
&< &\varepsilon(s)-b_-+(S-s)a_+
\\
&<& b_+-b_-=\kappa.
\end{eqnarray*}
\end{proof}

\begin{pro}
There exists $s_0$ such that, for $s<s_0$,
$\{(u_i,\bar{u}_i)\}
\subset \mathcal{M}_{\rho_+}^d(\bgamma,\bdelta)$
for $\bgamma\in a_s$
has a subsequence which converges to a broken gradient/continuation trajectory
without bubble tree in the sense of Floer--Gromov convergence.
\end{pro}

\begin{proof}
By Lemma 4.2, $E(u_i)$ is uniformly bounded, and
the Floer--Gromov compactness theorem implies that
$\{(u_i,\bar{u}_i)\}$ has a subsequence 
which converges to a broken gradient/continuation trajectory 
$((v_1,\bar{v}_1),\ldots,(v_N,\bar{v}_N))
\in \hat{\mathcal{M}}_s(\bgamma,\btheta_1)\times\cdots
\times \hat{\mathcal{M}}_s(\btheta_{i-1},\btheta_i)
\times \mathcal{M}_{\rho_+}(\btheta_i,\btheta_{i+1})
\times \hat{\mathcal{M}}_S(\btheta_{i+1},\btheta_{i+2})\times\cdots
\times\hat{\mathcal{M}}_S(\btheta_{N-1},\bdelta)$
with bubble trees; the tail components of the bubble trees are
\begin{itemize}
\item[(i)] a pseudoholomorphic sphere $v:S^2\to M$, 
\item[(ii)] a pseudoholomorphic disc
$v:D^2\to M$ with $\bar{v}:\partial D^2\to L$ such that $v|_{\partial D^2}=\iota\circ \bar{v}$,
\item[(iii)] a pseudoholomorphic disc $v:D^2\to M$ of $(v,\bar{v})\in K(\iota)$.
\end{itemize}
But, since our Lagrangian immersion is exact, the bubbles of (i) and (ii) can not occur.
Moreover, by Lemma 4.3, there exists $s_0$ such that for $s<s_0$
the symplectic area of the bubble trees is less than or equal to 
$\kappa\ (<\sigma)$, and the bubbles of (iii) can not occur.
Hence there is no bubble tree and
the subsequence converges to the broken gradient/continuation trajectory.
\end{proof}

To define our continuations, we use the following compactness theorems:

\begin{thm} 
There exists $s_0$ such that, for $s<s_0$,
$\mathcal{M}_{\rho_+}^{d,0}(\bgamma,\bdelta)$ for $\bgamma\in a_s$
is compact.
\end{thm}

\begin{proof}
Let $s_0$ be as in Proposition 4.4, and $s<s_0$.
Suppose on the contrary that $\mathcal{M}_{\rho_+}^{d,0}(\bgamma,\bdelta)$ 
is not compact.
Then there exists a sequence
$\{(u_i,\bar{u}_i)\}
\subset \mathcal{M}_{\rho_+}^{d,0}(\bgamma,\bdelta)$ 
such that any subsequence does not converge in 
$\mathcal{M}_{\rho_+}^{d,0}(\bgamma,\bdelta)$.
On the other hand, 
from Proposition 4.4,
$\{(u_i,\bar{u}_i)\}$ has a subsequence
converges to a broken gradient/continuation trajectory 
$((v_1,\bar{v}_1),\ldots,(v_N,\bar{v}_N))$ without bubble tree.
In this case, $N$ turns out to be 1 and 
$(v_1,\bar{v}_1)\in\mathcal{M}_{\rho_+}^0(\bgamma,\bdelta)$
for generic $\{J_t\}_{t\in[0,1]}$
because of the virtual dimension counting.
Since the subsequence preserves the condition 
$F_s(\bgamma)-F_S(\bdelta)=f_s(\bgamma)-f_S(\bdelta)$,
the limit $(v_1,\bar{v}_1)$ is in $\mathcal{M}_{\rho_+}^{d,0}(\bgamma,\bdelta)$,
which contradicts that any subsequence does not converge in 
$\mathcal{M}_{\rho_+}^{d,0}(\bgamma,\bdelta)$. 
Thus $\mathcal{M}_{\rho_+}^{d,0}(\bgamma,\bdelta)$ is compact.
\end{proof}

\begin{thm} 
There exists $s_0$ such that, for $s<s_0$,
$\mathcal{M}_{\rho_+}^{d,1}(\bgamma,\bdelta)$ for $\bgamma\in a_s$
has a suitable compactification whose boundary is given by
\[
\bigcup_{\btheta\in a_s}
\hat{\mathcal{M}}_s^{d,0}(\bgamma,\btheta)
\times \mathcal{M}_{\rho_+}^{d,0}(\btheta,\bdelta)
\cup
\bigcup_{\btheta\in c_S}
\mathcal{M}_{\rho_+}^{d,0}(\bgamma,\btheta)
\times \hat{\mathcal{M}}_S^{d,0}(\btheta,\bdelta).
\]
\end{thm}

\begin{proof}
The proof is based on the standard gluing-compsctness argument.
For a pair $((u_1,\bar{u}_1), (u_2,\bar{u}_2))
\in \hat{\mathcal{M}}_s^{d,0}(\bgamma,\btheta)\times\mathcal{M}_{\rho_+}^{d,0}(\btheta,\bdelta)$,
the gluing procedure gives a unique connected component of 
$\mathcal{M}_{\rho_+}^1(\bgamma,\bdelta)$, say $\mathcal{M}$,
such that the pair is a compactifying point of $\mathcal{M}$.
Since 
\begin{eqnarray*}
F_s(\bgamma)-F_S(\bdelta)&=&
F_s(\bgamma)-F_s(\btheta)+F_s(\btheta)-F_S(\bdelta)
\\
&=&
f_s(\bgamma)-f_s(\btheta)+f_s(\btheta)-f_S(\bdelta)
\\
&=&
f_s(\bgamma)-f_S(\bdelta),
\end{eqnarray*}
$\mathcal{M}$ is contained in $\mathcal{M}_{\rho_+}^{d,1}(\bgamma,\bdelta)$.
Similarly, we can glue the pairs of 
$\mathcal{M}_{\rho_+}^{d,0}(\bgamma,\btheta)\times\hat{\mathcal{M}}_S^{d,0}(\btheta,\bdelta)$
to make connected components of $\mathcal{M}_{\rho_+}^{d,1}(\bgamma,\bdelta)$.

On the other hand, from Proposition 4.4,
let $\{(u_i,\bar{u}_i)\}\subset \mathcal{M}^{d,1}_{\rho_+}(\bgamma,\bdelta)$
be a sequence which converges to a broken gradient/continuation trajectory
$((v_1,\bar{v}_1),\ldots,(v_N,\bar{v}_N))$ without bubble tree.
In this case, $N$ turns out to be 2 
and $((v_1,\bar{v}_1),(v_2,\bar{v}_2))\in
\hat{\mathcal{M}}_s^0(\bgamma,\btheta)\times\mathcal{M}_{\rho_+}^0(\btheta,\bdelta)$
or 
$\mathcal{M}_{\rho_+}^0(\bgamma,\btheta)\times\hat{\mathcal{M}}_S^0(\btheta,\bdelta)$
for generic $\{J_t\}_{t\in [0,1]}$ because of the virtual dimension counting.

Suppose $((v_1,\bar{v}_1),(v_2,\bar{v}_2))\in
\hat{\mathcal{M}}_s^0(\bgamma,\btheta)\times\mathcal{M}_{\rho_+}^0(\btheta,\bdelta)$
is the limit of the sequence $\{(u_i,\bar{u}_i)\}\subset \mathcal{M}_{\rho_+}^{d,1}(\bgamma,\bdelta)$.
By Lemma 4.3, there exists $s_0$ such that, for $s<s_0$, $E(u_i)<\kappa$;
since $0<E(v_1)$, $0<E(v_2)$ and $E(v_1)+E(v_2)\leq \limsup E(u_i)$,
\[
E(v_1)=F_s(\bgamma)-F_s(\btheta)<\kappa.
\]
Then, by Theorem 3.8, there exists $s_1$ such that, for $s<s_1$, 
we have $\btheta\in a_s$.
Moreover, since $\bgamma,\btheta\in a_s$,
there exists $s_2$ such that, for $s<s_2$, 
$F_s(\bgamma)=f_s(\bgamma)$ and $F_s(\btheta)=f_s(\btheta)$;
and 
$F_s(\bgamma)-F_s(\btheta)=f_s(\bgamma)-f_s(\btheta)$.
Since $(u_i,\bar{u}_i)$ is distinguished, i.e.
$F_s(\bgamma)-F_S(\bdelta)=f_s(\bgamma)-f_S(\bdelta)$,
\begin{eqnarray*}
F_s(\btheta)-F_S(\bdelta) &=& 
F_s(\btheta)-F_s(\bgamma)+F_s(\bgamma)-F_S(\bdelta)
\\
&=&
f_s(\btheta)-f_s(\bgamma)+f_s(\bgamma)-f_S(\bdelta)
\\
&=& f_s(\btheta)-f_S(\bdelta).
\end{eqnarray*}
Thus $((v_1,\bar{v}_1),(v_2,\bar{v}_2))\in
\hat{\mathcal{M}}_s^{d,0}(\bgamma,\btheta)\times\mathcal{M}_{\rho_+}^{d,0}(\btheta,\bdelta)$.

Suppose $((v_1,\bar{v}_1),(v_2,\bar{v}_2))\in
\mathcal{M}_{\rho_+}^0(\bgamma,\btheta)\times\hat{\mathcal{M}}_S^0(\btheta,\bdelta)$
is the limit of the sequence $\{(u_i,\bar{u}_i)\}\subset \mathcal{M}_{\rho_+}^{d,1}(\bgamma,\bdelta)$.
Note that there exists $s_0$ such that, for $s<s_0$, 
$\varepsilon(s)<b_+-a_+$ and $|f_s(\bgamma)|<\varepsilon(s)$ for $\bgamma\in a_s$.
Then, by Lemma 4.2, for $s<s_0$, 
\begin{eqnarray}
F_s(\bgamma)-F_S(\btheta) &=&E(v_1)+\int_{-\infty}^{\infty}\frac{d\rho_+(\tau)}{d\tau}
\int_0^1H(t,v_1(\tau,t))dtd\tau \nonumber
\\ 
&>& (S-s)\int_0^1\min_{x\in M}H(t,x)dt \nonumber
\\
&\geq&-a_+ \nonumber
\\
&=& b_+-a_+-b_--\kappa \nonumber
\\
&>& \varepsilon(s)-b_--\kappa \nonumber
\\
&>&f_s(\bgamma)-f_S(\btheta)-\kappa.
\end{eqnarray}
Since $(u_i,\bar{u}_i)$ is distinguished, i.e.
$F_s(\bgamma)-F_S(\bdelta)=f_s(\bgamma)-f_S(\bdelta)$,
there exists $m\in\mathbb{Z}$ such that
\begin{eqnarray}
F_s(\bgamma)-F_S(\btheta) &=&f_s(\bgamma)-f_S(\btheta)+m\kappa,
\\
F_S(\btheta)-F_S(\bdelta) &=& f_S(\btheta)-f_S(\bdelta)-m\kappa.
\end{eqnarray}
From (3) and (4), we obtain $m\geq 0$;
and from (5), since $F_S(\btheta)-F_S(\bdelta)=E(v_2)>0$ 
and $\kappa>f_S(\btheta)-f_S(\bdelta)$, we obtain $0\geq m$.
Thus $m=0$,
which implies $((v_1,\bar{v}_1),(v_2,\bar{v}_2))\in 
\mathcal{M}_{\rho_+}^{d,0}(\bgamma,\btheta)\times \hat{\mathcal{M}}_S^{d,0}(\btheta,\bdelta)$.
We obtain the compactification of $\mathcal{M}_{\rho_+}^{d,1}(\bgamma,\bdelta)$.
\end{proof}

Similarly, we can prove the following theorems:

\begin{thm} 
There exists $s_0$ such that, for $s<s_0$,
$\mathcal{M}_{\rho_-}^{d,0}(\bgamma,\bdelta)$ for $\bdelta\in a_s$ is compact.
\end{thm}

\begin{thm}
There exists $s_0$ such that, for $s<s_0$,
$\mathcal{M}_{\rho_-}^{d,1}(\bgamma,\bdelta)$ for $\bdelta\in a_s$
has a suitable compactification whose boundary is given by
\[
\bigcup_{\btheta\in c_S}
\hat{\mathcal{M}}_S^{d,0}(\bgamma,\btheta)
\times \mathcal{M}_{\rho_-}^{d,0}(\btheta,\bdelta)
\cup
\bigcup_{\btheta\in a_s}
\mathcal{M}_{\rho_-}^{d,0}(\bgamma,\btheta)
\times \hat{\mathcal{M}}_s^{d,0}(\btheta,\bdelta).
\]
\end{thm}

For $s<s_0$, Theorem 4.5 and 4.7 allow us to define linear maps $\Phi_+:A_s\to C_S$ 
and $\Phi_-:C_S\to A_s$ by
\[
\Phi_+(\bgamma):=\sum_{\bdelta\in c_S}\sharp\mathcal{M}_{\rho_+}^{d,0}(\bgamma,\bdelta)\bdelta
\]
for $\bgamma\in a_s$, and 
\[
\Phi_-(\bgamma):=\sum_{\bdelta\in a_s}\sharp\mathcal{M}_{\rho_-}^{d,0}(\bgamma,\bdelta)\bdelta
\]
for $\bgamma\in c_S$.
We call $\Phi_+$ and $\Phi_-$ {\it continuations}.
Then Corollary 3.11 and Theorem 4.6 and 4.8 imply the following theorem:

\begin{thm}
For $s<s_0$,
$\Phi_+\circ\partial_s=\partial_S\circ\Phi_+$
and
$\Phi_-\circ\partial_S=\partial_s\circ\Phi_-$.
\end{thm}

Thus, for $s<s_0$, $\Phi_+$ and $\Phi_-$ induce homomorphisms
$\Phi_{+*}:H(A_s,\partial_s)\to H(C_S,\partial_S)$
and 
$\Phi_{-*}:H(C_S,\partial_S)\to H(A_s,\partial_s)$, respectively.

\section{\bf Homotopies of continuations}

Let $\brho=\brho(w,\tau):[0,\infty)\times\mathbb{R}\to \mathbb{R}$ be a smooth function
such that 
\begin{itemize}
\item $\frac{\partial \brho}{\partial \tau}\geq0$ when $\tau<0$, 
and $\frac{\partial \brho}{\partial \tau}\leq 0$ when $\tau>0$,
\item $w\mapsto\brho(w,0)$ is a monotone map onto $[s,S]$,
\item $\brho(0,\tau)\equiv s$,
\item for $w$ large enough,
\[
\brho(w,\tau)=\begin{cases}
\rho_+(\tau+w) & \tau\leq 0,
\\
\rho_-(\tau-w) & \tau\geq 0. 
\end{cases}
\]
\end{itemize}
Let $\rho_w(\tau)$ denote $\brho(w,\tau)$.

Let $s\in T$. We define $\mathcal{M}_{\brho}(\bgamma,\bdelta)$ for $\bgamma,\bdelta\in c_s$ by
\[
\mathcal{M}_{\brho}(\bgamma,\bdelta):=
\left\{(w,(u,\bar{u})):
w\in[0,\infty)
\mbox{ and }
(u,\bar{u})\in\mathcal{M}_{\rho_w}(\bgamma,\bdelta)
\right\}.
\]

\begin{thm}
For generic $\{J_t\}_{t\in[0,1]}$, $\mathcal{M}_{\brho}(\bgamma,\bdelta)$ is a finite dimensional 
smooth manifold with boundary $\partial \mathcal{M}_{\brho}(\bgamma,\bdelta)
=\{(0,(u,\bar{u}))\in \mathcal{M}_{\brho}(\bgamma,\bdelta)\}$.
\end{thm}

Let $\mathcal{M}_{\brho}^k(\bgamma,\bdelta)$ denote the $k$-dimensional component of 
$\mathcal{M}_{\brho}(\bgamma,\bdelta)$.
Moreover, we define $\mathcal{M}_{\brho}^d(\bgamma,\bdelta)$ by
\[
\mathcal{M}_{\brho}^d(\bgamma,\bdelta):=
\left\{(w,(u,\bar{u})):
w\in[0,\infty)
\mbox{ and }
(u,\bar{u})\in\mathcal{M}_{\rho_w}^d(\bgamma,\bdelta)
\right\}.
\]
Let $\mathcal{M}_{\brho}^{d,k}(\bgamma,\bdelta)$ denote the $k$-dimensional component of 
$\mathcal{M}_{\brho}^d(\bgamma,\bdelta)$.

\begin{lem}
There exists $s_0$ such that for $s<s_0$
\[
E(u)<\kappa
\]
for $\bgamma,\bdelta\in a_s$ and $(u,\bar{u})\in \mathcal{M}_{\brho}^d(\bgamma,\bdelta)$.
\end{lem}

\begin{proof}
By Lemma 4.2, we obtain
\begin{eqnarray*}
E(u)&=&F_s(\bgamma)-F_s(\bdelta)
\\
&&
-\int_{-\infty}^0\frac{\partial \rho_w(\tau)}{\partial \tau}\int_0^1H(t,u(\tau,t))dtd\tau
\\
&&
-\int_0^{\infty}\frac{\partial \rho_w(\tau)}{\partial \tau}\int_0^1H(t,u(\tau,t))dtd\tau
\\
&\leq&
F_s(\bgamma)-F_s(\bdelta)+(S-s)(a_+-a_-).
\end{eqnarray*}
Note that there exists $s_0$ such that, for $s<s_0$,
$2\varepsilon(s)<\kappa-\|H\|$ and $|f_s(\bgamma)|<\varepsilon(s)$ for $\bgamma\in a_s$
(and $|f_s(\bdelta)|<\varepsilon(s)$ for $\bdelta\in a_s$).
Since $(u,\bar{u})$ is distinguished, for $s<s_0$,
\begin{eqnarray*}
E(u)&\leq& F_s(\bgamma)-F_s(\bdelta)+(S-s)(a_+-a_-)
\\
&=&
f_s(\bgamma)-f_s(\bdelta)+(S-s)(a_+-a_-)
\\
&\leq &2\varepsilon(s)+(S-s)(a_+-a_-)
\\
&<& \kappa.
\end{eqnarray*}
\end{proof}

\begin{pro}
There exists $s_0$ such that, for $s<s_0$,
$\{(w_i,(u_i,\bar{u}_i))\}\subset \mathcal{M}_{\brho}^d(\bgamma,\bdelta)$ 
for $\bgamma,\bdelta\in a_s$
has a subsequece which converges to 
a broken gradient/continuation trajectory without bubble tree
in the sense of Floer--Gromov convergence.
\end{pro}

\begin{proof}
By Lemma 4.2, $E(u_i)$ is uniformly bounded,
and the Floer--Gromov compactness theorem implies that $\{(w_i,(u_i,\bar{u}_i))\}$
has a subsequence which converges to a gradient/continuation trajectory of
\begin{itemize}
\item[(A)] $\partial\mathcal{M}_{\brho}(\bgamma,\bdelta)
=\{(0,(u,\bar{u}))\in\mathcal{M}_{\brho}(\bgamma,\bdelta)\}$ when $\lim w_i=0$,
\item[(B)] $\hat{\mathcal{M}}_s(\bgamma,\btheta_1)\times\cdots\times
\hat{\mathcal{M}}_s(\btheta_{i-1},\btheta_i)\times
\mathcal{M}_{\rho_{w}}(\btheta_i,\btheta_{i+1})\times
\hat{\mathcal{M}}_s(\btheta_{i+1},\btheta_{i+1})\times\cdots\times
\hat{\mathcal{M}}_s(\btheta_{N-1},\bdelta)$ when $\lim w_i=w\in (0,\infty)$,
\item[(C)] $\hat{\mathcal{M}}_s(\bgamma,\btheta_1)\times\cdots\times
\hat{\mathcal{M}}_s(\btheta_{i-1},\btheta_i)\times
\mathcal{M}_{\rho_+}(\btheta_i,\btheta_{i+1})\times
\hat{\mathcal{M}}_S(\theta_{i+1},\btheta_{i+2})\times\cdots\times
\hat{\mathcal{M}}_S(\btheta_{j-1},\btheta_j)\times
\mathcal{M}_{\rho_-}(\btheta_j,\btheta_{j+1})\times
\hat{\mathcal{M}}_s(\btheta_{j+1},\btheta_{j+2})\times\cdots\times 
\hat{\mathcal{M}}_s(\btheta_{N-1}, \\
\bdelta)$ when $\lim w_i=\infty$
\end{itemize}
with bubble trees; the tail components of the bubble trees are
\begin{itemize}
\item[(i)] a pseudoholomorphic sphere $v:S^2\to M$, 
\item[(ii)] a pseudoholomorphic disc
$v:D^2\to M$ with $\bar{v}:\partial D^2\to L$ such that $v|_{\partial D^2}=\iota\circ \bar{v}$,
\item[(iii)] a pseudoholomorphic disc $v:D^2\to M$ of $(v,\bar{v})\in K(\iota)$.
\end{itemize}
But, since our Lagrangian immersion is exact, the bubbles of (i) and (ii) can not occur.
Moreover, by Lemma 5.2, there exists $s_0$ such that for $s<s_0$
the symplectic area of the bubble trees is less than or equal to $\kappa\ (<\sigma)$, 
and the bubbles of (iii) can not occur.
Hence there is no bubble tree and
the subsequence converges to the broken gradient/continuation trajectory.
\end{proof}

To define our homotopy, we need the following compactness theorems:

\begin{thm} 
There exists $s_0$ such that, for $s<s_0$, 
$\mathcal{M}_{\brho}^{d,0}(\bgamma,\bdelta)$ for $\bgamma,\bdelta\in a_s$ is compact.
\end{thm}

\begin{proof}
Let $s_0$ be as in Proposition 5.3, and $s<s_0$.
Suppose on the contrary that $\mathcal{M}_{\brho}^{d,0}(\bgamma,\bdelta)$ 
is not compact.
Then there exists a sequence
$\{(w_i,(u_i,\bar{u}_i))\} \\
\subset \mathcal{M}_{\brho}^{d,0}(\bgamma,\bdelta)$ 
such that any subsequence does not converge in 
$\mathcal{M}_{\brho}^{d,0}(\bgamma,\bdelta)$.
On the other hand, from Proposition 5.3,
$\{(w_i,(u_i,\bar{u}_i))\}$ has a subsequence which
converges to a broken gradient/continuation trajectory 
$((w,(v_1,\bar{v}_1)), \\
\ldots,(w, (v_N,\bar{v}_N)))$ without bubble tree.
In this case, $N$ turns out to be 1 and 
$(w,(v_1,\bar{v}_1))\in\mathcal{M}_{\brho}^0(\bgamma,\bdelta)$
for generic $\{J_t\}_{t\in[0,1]}$
because of the virtual dimension counting.
Since the subsequence preserves the condition 
$F_s(\bgamma)-F_s(\bdelta)=f_s(\bgamma)-f_s(\bdelta)$,
the limit $(w,(v_1,\bar{v}_1))$ is in
$\mathcal{M}_{\brho}^{d,0}(\bgamma,\bdelta)$,
which contradicts that any subsequence does not converge in 
$\mathcal{M}_{\brho}^{d,0}(\bgamma,\bdelta)$. 
Thus $\mathcal{M}_{\brho}^{d,0}(\bgamma,\bdelta)$ is compact.
\end{proof}

\begin{thm} 
There exists $s_0$ such that, for $s<s_0$, 
$\mathcal{M}_{\brho}^{d,1}(\bgamma,\bdelta)$ 
for $\bgamma,\bdelta\in a_s$
has a suitable compactification whose boundary consists of
\begin{itemize}
\item[(A)] $\partial\mathcal{M}_{\brho}^{d,1}(\bgamma,\bdelta)=\{(0,(u,\bar{u}))\in
\mathcal{M}_{\brho}^{d,1}(\bgamma,\bdelta)\}$,
\item[(B1)] $\bigcup_{\btheta\in a_s}
\hat{\mathcal{M}}_s^{d,0}(\bgamma,\btheta)
\times \mathcal{M}_{\brho}^{d,0}(\btheta,\bdelta)$,
\item[(B2)] $\bigcup_{\btheta\in a_s}
\mathcal{M}_{\brho}^{d,0}(\bgamma,\btheta)
\times \hat{\mathcal{M}}_s^{d,0}(\btheta,\bdelta)$,
\item[(C)] $\bigcup_{\btheta\in c_S}\mathcal{M}_{\rho_+}^{d,0}(\bgamma,\btheta)\times
\mathcal{M}_{\rho_-}^{d,0}(\btheta,\bdelta)$.
\end{itemize}
\end{thm}

Note that, if $\bgamma=\bdelta$, then
\[
\partial \mathcal{M}_{\brho}^{d,1}(\bgamma,\bgamma)=\{(u,\bar{u}):
u(\tau,t)=\gamma(t)\mbox{ and }\bar{u}(i)=\bar{\gamma}(i) \mbox{ for }=0,1\},
\]
and, if $\bgamma\neq \bdelta$, then $\partial \mathcal{M}_{\brho}^{d,1}(\bgamma,\bdelta)=\emptyset$
since the virtual dimension equals $-1$.

\begin{proof}
First, $\mathcal{M}_{\brho}^{d,1}(\bgamma,\bdelta)$ has the boundary
$\partial\mathcal{M}_{\brho}^{d,1}(\bgamma,\bdelta)$,
which is (A).
Otherwise, we compactify $\mathcal{M}_{\brho}^{d,1}(\bgamma,\bdelta)$
by the standard gluing-compsctness argument. 
For a pair $((u_1,\bar{u}_1), (w,(u_2,\bar{u}_2)))
\in \hat{\mathcal{M}}_s^{d,0}(\bgamma,\btheta)\times
\mathcal{M}_{\brho}^{d,0}(\btheta,\bdelta)$, which is (B1),
there exists a unique connected component of 
$\mathcal{M}_{\brho}^1(\bgamma,\bdelta)$, say $\mathcal{M}$,
such that the pair is a compactifying point of $\mathcal{M}$.
Since 
\begin{eqnarray*}
F_s(\bgamma)-F_s(\bdelta)&=&
F_s(\bgamma)-F_s(\btheta)+F_s(\btheta)-F_s(\bdelta)
\\
&=&
f_s(\bgamma)-f_s(\btheta)+f_s(\btheta)-f_s(\bdelta)
\\
&=&
f_s(\bgamma)-f_s(\bdelta),
\end{eqnarray*}
$\mathcal{M}$ is contained in $\mathcal{M}_{\brho}^{d,1}(\bgamma,\bdelta)$.
Similarly, 
we can glue the pairs of (B2) and (C) to make connected components of 
$\mathcal{M}_{\brho}^{d,1}(\bgamma,\bdelta)$.

On the other hand, from Proposition 5.3,
let $\{(w_i,(u_i,\bar{u}_i))\}\subset \mathcal{M}^{d,1}_{\brho}(\bgamma,\bdelta)$
be a sequence which converges to a broken gradient/continuation trajectory
$((w,(v_1,\bar{v}_1)),\ldots,(w,(v_N,\bar{v}_N)))$ of (B) or (C) without bubble tree.
In this case, $N$ turns out to be 2 
and $((w,(v_1,\bar{v}_1)), (w,(v_2,\bar{v}_2)))$ is in 
$\hat{\mathcal{M}}_s^0(\bgamma,\btheta)\times\mathcal{M}_{\brho}^0(\btheta,\bdelta)$,
$\mathcal{M}_{\brho}^0(\bgamma,\btheta)\times\hat{\mathcal{M}}_s^0(\btheta,\bdelta)$ or
$\mathcal{M}_{\rho+}^0(\bgamma,\btheta)\times\mathcal{M}_{\rho_-}^0(\btheta,\bdelta)$
for generic $\{J_t\}_{t\in [0,1]}$ because of the virtual dimension counting.

Suppose $((v_1,\bar{v}_1),(w,(v_2,\bar{v}_2)))\in
\hat{\mathcal{M}}_s^0(\bgamma,\btheta)\times\mathcal{M}_{\brho}^0(\btheta,\bdelta)$
is the limit of the sequence 
$\{(w_i,(u_i,\bar{u}_i))\}\subset \mathcal{M}_{\brho}^{d,1}(\bgamma,\bdelta)$.
By Lemma 5.2, there exists $s_0$ such that, for $s<s_0$, $E(u_i)<\kappa$;
since $0<E(v_1)$, $0<E(v_2)$ and $E(v_1)+E(v_2)\leq \limsup E(u_i)\leq \kappa$,
\[
E(v_1)=F_s(\bgamma)-F_s(\btheta)<\kappa.
\]
Then, by Theorem 3.8, there exists $s_1$ such that, for $s<s_1$,
we have $\btheta\in a_s$.
Moreover, since $\bgamma,\btheta,\bdelta\in a_s$, there exists $s_2$ such that, for $s<s_2$,
$F_s(\bgamma)=f_s(\bgamma)$, $F_s(\btheta)=f_s(\btheta)$ and $F_s(\bdelta)=f_s(\bdelta)$.
Thus $((v_1,\bar{v}_1),(w,(v_2,\bar{v}_2)))\in
\hat{\mathcal{M}}_s^{d,0}(\bgamma,\btheta)\times\mathcal{M}_{\brho}^{d,0}(\btheta,\bdelta)$, 
which is (B1).

Suppose $((w,(v_1,\bar{v}_1)),(v_2,\bar{v}_2))\in
\hat{\mathcal{M}}_{\brho}^0(\bgamma,\btheta)\times\mathcal{M}_s^0(\btheta,\bdelta)$
is the limit of the sequence 
$\{(w_i,(u_i,\bar{u}_i))\}\subset \mathcal{M}_{\brho}^{d,1}(\bgamma,\bdelta)$.
By Lemma 5.2, there exists $s_0$ such that, for $s<s_0$, $E(u_i)<\kappa$;
since $0<E(v_1)$, $0<E(v_2)$ and $E(v_1)+E(v_2)\leq \limsup E(u_i)\leq \kappa$,
\[
E(v_2)=F_s(\btheta)-F_s(\bdelta)<\kappa.
\]
Then, by Theorem 3.10, there exists $s_1$ such that, for $s<s_1$,
we have $\btheta\in a_s$.
Moreover, since $\bgamma,\btheta,\bdelta\in a_s$, there exists $s_2$ such that, for $s<s_2$,
$F_s(\bgamma)=f_s(\bgamma)$, $F_s(\btheta)=f_s(\btheta)$ and $F_s(\bdelta)=f_s(\bdelta)$.
Thus $((w,(v_1,\bar{v}_1)),(v_2,\bar{v}_2))\in
\hat{\mathcal{M}}_{\brho}^{d,0}(\bgamma,\btheta)\times\mathcal{M}_s^{d,0}(\btheta,\bdelta)$, 
which is (B2).

Finally, suppose $((v_1,\bar{v}_1),(v_2,\bar{v}_2))\in
\mathcal{M}_{\rho_+}^0(\bgamma,\btheta)\times\mathcal{M}_{\rho_-}^0(\btheta,\bdelta)$
is the limit of the sequence 
$\{(w_i,(u_i,\bar{u}_i))\}\subset \mathcal{M}_{\brho}^{d,1}(\bgamma,\bdelta)$.
Note that there exists $s_0$ such that, for $s<s_0$, 
$\varepsilon(s)<\min\{b_+-a_+,a_--b_-\}$
and $|f_s(\bgamma)|<\varepsilon(s)$ for $\bgamma\in a_s$.
Then, by Lemma 4.2, for $s<s_0$,
\begin{eqnarray}
F_s(\bgamma)-F_S(\btheta)
&=&E(v_1)+\int_{-\infty}^{\infty}\frac{\partial \rho_+(\tau)}{\partial \tau}\int_0^1H(t,v_1(\tau,t))dtd\tau
\nonumber
\\
&>&(S-s)\int_0^1\min_{x\in M}H(t,x)dt \nonumber
\\
&\geq& -a_+ \nonumber
\\
&=& b_+-a_+-b_--\kappa \nonumber
\\
&>& f_s(\bgamma)-f_S(\btheta)-\kappa,
\end{eqnarray}
and similarly, 
\begin{eqnarray}
F_S(\btheta)-F_s(\bdelta)
&=&E(v_2)+\int_{-\infty}^{\infty}\frac{\partial \rho_-(\tau)}{\partial \tau}\int_0^1H(t,v_1(\tau,t))dtd\tau
\nonumber
\\
&>&-(S-s)\int_0^1\max_{x\in M}H(t,x)dt \nonumber
\\
&\geq& a_- \nonumber
\\
&=& b_++a_--b_--\kappa \nonumber
\\
&>& f_S(\btheta)-f_s(\bdelta)-\kappa.
\end{eqnarray}
Since $(u_i,\bar{u}_i)$ is distinguished, i.e.
$F_s(\bgamma)-F_s(\bdelta)=f_s(\bgamma)-f_s(\bdelta)$,
there exists $m\in\mathbb{Z}$ such that
\begin{eqnarray}
F_s(\bgamma)-F_S(\btheta) &=&f_s(\bgamma)-f_S(\btheta)+m\kappa,
\\
F_S(\btheta)-F_s(\bdelta) &=& f_S(\btheta)-f_s(\bdelta)-m\kappa.
\end{eqnarray}
From (6) and (8), we obtain $m\geq 0$; and from (7) and (9) we obtain $0\geq m$.
Thus $m=0$,
which implies $((v_1,\bar{v}_1),(v_2,\bar{v}_2))\in 
\mathcal{M}_{\rho_+}^{d,0}(\bgamma,\btheta)\times \mathcal{M}_{\rho_-}^{d,0}(\btheta,\bdelta)$,
which is (C).
We obtain the compactification of $\mathcal{M}_{\brho}^{d,1}(\bgamma,\bdelta)$.
\end{proof}

For $s<s_0$, Theorem 5.4 allows us to define a linear map 
$H_s:A_s\to A_s$~by
\[
H_s(\bgamma):=\sum_{\bdelta\in a_s}\sharp\mathcal{M}_{\brho}^{d,0}(\bgamma,\bdelta)\bdelta
\]
for $\bgamma\in a_s$.
We call $H_s$ a {\it homotopy of continuations}.
Then Corollary 3.12 and Theorem 5.5 imply the following theorem:

\begin{thm}
For $s<s_0$,
\[
{\rm id}+H_s\circ \partial_s+\partial_s\circ H_s+\Phi_-\circ \Phi_+=0.
\]
\end{thm}

\begin{cor}
For $s<s_0$,
\[
\Phi_{-*}\circ\Phi_{+*}={\rm id}:H(A_s,\partial_s)\to H(A_s,\partial_s), 
\]
and hence $\Phi_{+*}:H(A_s,\partial_s)\to H(C_S,\partial_S)$ is injective.
\end{cor}

\section{The proof of Theorem 1.1}

Let $(M,\omega)$ be a closed symplectic manifold
or a non-compact symplectic manifold with convex end,
and $\iota:L\to M$ an exact Lagrangian immersion from a closed manifold $L$.
Suppose the non-injective points of $\iota: L\to M$ are transverse.
Let $\varphi^H_1$ be the time one map generated by $X_H$.

Note that, for $\gamma(t):=\varphi^H_t(\delta(t))$, $(\gamma,\bar{\gamma})\in c_1$
if and only if 
\[
\delta(t)\equiv p\in \iota(L)\cap(\varphi^{sH}_1)^{-1}(\iota(L)).
\]
Hence we can identify $(\gamma,\bar{\gamma})\in c_1$ with $(x,x')\in L\times L$
such that $\iota(x)=((\varphi^H_1)^{-1}\circ\iota)(x')$.
Suppose $\iota:L\to M$ and $\varphi^H_1\circ \iota:L\to M$ intersect transversely.
Then we slightly perturb $H$, if necessary, and choose $S\in T$ so that 
the numbers of the elements of
$\{(x,x')\in L\times L:\iota(x)=((\varphi^H_1)^{-1}\circ \iota)(x')\}$
and $c_S$ are equal.
First, our distinguished Floer homology gives
\[
\sharp c_S\geq \dim H(C_S,\partial_S).
\]
Secondly, from Theorem 3.14, there exists $s_0$ such that for $s<s_0$ and $s\in T$
\[
\dim H(A_s,\partial_s)=\sum_{k=0}^{\dim L}\dim H_k(L;\mathbb{Z}_2).
\]
Suppose $\|H\|<\kappa$. Then, from Corollary 5.7, there exists $s_0$ such that, 
for $s<s_0$ and $s\in T$,
$\Phi_{+*}:H(A_s,\partial_s)\to H(C_S,\partial_S)$ is injective. Thus
\[
\dim H(C_S,\partial_S)\geq \dim H(A_s,\partial_s).
\]
Therefore we obtain
\[
\sharp \left\{(x,x')\in L\times L: \iota(x)=((\varphi^H_1)^{-1}\circ\iota)(x')\right\}
\geq \sum_{k=0}^{\dim L}\dim H_k(L;\mathbb{Z}_2).
\]


\begin{thebibliography}{99}
\bibitem{ak}
M. Akaho,
Intersection theory for Lagrangian immersions.
Math. Res. Lett. 12 (2005), no. 4, 543-550.
\bibitem{ch1}
Yu. V. Chekanov,
Hofer's symplectic energy and Lagrangain intersections.
Contact and symplectic geometry (Cambridge, 1994), 296--306, 
Publ. Newton Inst., 8, Cambridge Univ. Press, Cambridge, 1996.
\bibitem{ch2}
Yu. V. Chekanov,
Lagrangian intersections, symplectic energy, and areas of holomorphic curves.
Duke Math. J. 95 (1998), no. 1, 213--226.
\bibitem{fl1}
A. Floer,
Morse theory for Lagrangian intersections.
J. Differential Geom. 28 (1988), no. 3, 513--547.
\bibitem{fl2}
A. Floer,
The unregularized gradient flow of the symplectic action.
Comm. Pure Appl. Math. 41 (1988), no. 6, 775--813.
\bibitem{fl3}
A. Floer,
Witten's complex and infinite-dimensional Morse theory.
J. Differential Geom. 30 (1989), no. 1, 207--221. 
\bibitem{gr}
M. Gromov, Pseudoholomorphic curves in symplectic manifolds.
Invent. Math. 82 (1985), no. 2, 307--347.
\bibitem{ho1}
H. Hofer,
Lagrangain embeddings and critical point theory.
Ann. Inst. H. Poincar\'{e} Anal. Non Lin\'{e}aire 2 (1985), no. 6, 407--462.
\bibitem{ho2}
H. Hofer,
On the topological properties of symplectic maps.
Proc. Roy. Soc. Edinburgh Sect. A 115 (1990), no. 1-2, 25--38.
\bibitem{hz}
H. Hofer and E. Zehnder,
Symplectic invariants and Hamiltonian dynamics.
Birkh\"{a}user Advanced Texts: Basler Lehrb\"{u}cher. Birkh\"{a}user Verlag, Basel, 1994. xiv+341 pp.
\bibitem{ms}
D. McDuff and D. Salamon,
$J$-holomorphic curves and symplectic topology. Second edition. 
American Mathematical Society Colloquium Publications, 52. 
American Mathematical Society, Providence, RI, 2012. xiv+726 pp.
\bibitem{oh1}
Y.-G. Oh,
Gromov--Floer theory and disjunction energy of compact Lagrangian embeddings.
Math. Res. Lett. 4 (1997), no. 6, 895--905.
\bibitem{po1}
L. Polterovich,
Symplectic displacement energy for Lagrangian submanifolds.
Ergodic Theory Dynam. Systems 13 (1993), no. 2, 357--367. 
\bibitem{sa1}
D. Salamon, 
Lectures on Floer homology. Symplectic geometry and topology (Park City, UT, 1997), 143--229, IAS/Park City Math. Ser., 7, Amer. Math. Soc., Providence, RI, 1999.
\end{thebibliography}
\end{document}